\documentclass[11pt]{amsart}
\usepackage{graphicx}
\usepackage{amssymb,amsfonts,amsthm}

\usepackage{amsmath,amsthm,amssymb}
\usepackage{amscd}
\usepackage{amsfonts}
\usepackage{color}
\title[Almost all 1-related groups are residually finite]{Almost all one-relator groups with at least three  generators are residually finite}

\author{Mark Sapir}
\address{Department of Mathematics, Vanderbilt University}
\email{mark.sapir@vanderbilt.edu}
\urladdr{www.math.vanderbilt.edu/$\sim$msapir}

\author{Iva \v Spakulov\' a}
\address{Department of Mathematics, Vanderbilt University}
\email{iva.kozakova@vanderbilt.edu}

\thanks{The work
was  supported in part by the NSF grant DMS 0700811 and by a BSF
(USA-Israeli) grant.}

\usepackage{amssymb}
\usepackage{amsmath}
\usepackage{amsthm}
\usepackage{amscd}
\long\def\comment#1\endcomment{}

\usepackage[letterpaper,tmargin=1.1in,bmargin=1.1in,lmargin=1.2in,rmargin=1.2in]{geometry}

\usepackage{color}

\newtheorem{theorem}{Theorem}[section]
\newtheorem{lemma}[theorem]{Lemma}

\theoremstyle{definition}
\newtheorem{definition}[theorem]{Definition}

\newtheorem{remark}[theorem]{Remark}

\newtheorem{question}[theorem]{Question}

\def\Z{{\mathbb Z}}
\def\R{{\mathbb R}}

\def\P{{\mathrm P}}

\begin{document}

\begin{abstract}
We prove that with probability tending to 1, a 1-relator group with at least 3 generators and the relator
of length $n$ is residually finite, virtually residually (finite $p$)-group for all
sufficiently large $p$, and coherent. The proof uses both combinatorial group
theory and non-trivial results about Brownian motions.
\end{abstract}

\maketitle

\section{Introduction}

Residual finiteness of 1-related groups is one of the main topics in
combinatorial group theory since 1960s. The first non-residually
finite examples  were given in \cite{BaumSol} (say, the Baumslag--Solitar
groups $BS(p,q)=\langle a, b\mid b^{-1}a^pb=b^q\rangle $
where $p$ and $q$ are different primes). Possibly the strongest
``positive'', non-probabilistic result so far is the result by D. Wise
\cite{Wise}: any one-relator group whose relator is a {\em positive word}
satisfying the condition $C\sp\prime (1/6)$ is residually finite.
The strongest ``negative" result appeared recently in the paper by Baumslag, Miller and Traeger
\cite{BMT}: Let $G=\langle a,b,\ldots\vert r=1\rangle$ be a
one-relator group with at least two generators and let
$G(r,w)=\langle a,b,\ldots\vert r\sp{r\sp w}=r\sp2\rangle$ where $w$
is an element of a free group with free generators $a,b,\dots$ that
does not commute with $r$. Then \cite[Theorem 1]{BMT} asserts that
the group $G(r,w)$ is not residually finite. Note that the length of
the relator of $G(r,w)$ is at most a constant multiple of the length
of the relator of $G$.

Another important problem about 1-related groups is whether every 1-related group is {\em coherent}, i.e. every finitely generated subgroup of it is finitely presented.

In this paper, we show that generically 1-related groups with at least 3 generators satisfy both properties: they are residually finite and coherent. We also discuss the case of 1-related groups with 2 generators at the end of the paper. 

We consider three natural models of choosing a random 1-related group.

{\bf Model NR}. For every $r\ge 0$, consider the set $S_r$ of all group words $R$ of length $r\ge 1$ in a free group $F_k=\langle x_1,...,x_k\rangle $. On that set, we choose the uniform probability measure.
By a {\em random 1-related group with $k$ generators of complexity $r$} we mean the group with presentation $\langle x_1,...,x_k\mid R=1 \rangle$ where $R$ is a random word from $S_r$.

{\bf Model CR}. In this model, we consider the set $CS_r$ of cyclically reduced words in $F_k$ of length $r$ and consider the uniform probability measure on that set. Then a {\em random 1-related group with $k$ generators of complexity $r$}  is a group $\langle x_1,...,x_k\mid R \rangle$ where $R$ is a random word from $CS_r$.

The third model is one considered, for example, in Kapovich-Schupp-Shpilrain \cite{KSS}.

{\bf Model IC}. Consider the equivalence relation on the set of all cyclically reduced words from $F_k$: two words are equivalent if the corresponding 1-related groups are isomorphic. In this model we consider any set $D_r$
of representatives of length exactly $r$ of all equivalence classes of words containing words of length $r$.
Consider the uniform probability measure on $D_r$. Then by a {\em a random 1-related group with $k$ generators of complexity $r$} we mean the group with presentation $\langle x_1,...,x_k\mid R=1 \rangle$ where $R$ is a random word from $D_r$.

Now given any property $P$ of groups and any of the three probabilistic models above,
consider the probability $p_r$ that a random $k$-generator
$1$-relator group of with relator of length $r$ has property $P$. If $p_r$ has a
limit $p$, we say that a random $k$-generator $1$-relator group has
this property with probability $p$.

We prove below (see
Lemmas \ref{lem2}, \ref{lem3}) that if the limit of probabilities $p_r$ exists in
the random model CR, then it coincides with the limit in
model NR and IC.

Recall that an ascending HNN extension of a free group (mapping torus of free group endomorphism) is a group of the form HNN$_\phi(F_k)=\langle x_1,...,x_k,t\mid x_1^t=\phi(x_1),...,x_k^t=\phi(x_k)\rangle$ where $\phi$ is an injective homomorphism of $F_k$.

Here is the main result of this paper. 

\begin{theorem}\label{thm3} A random $k$-generator $1$-relator group, $k\geq 3$,
can be embedded into an ascending HNN extension of a finitely
generated free group with probability $1$. Therefore almost
surely, such a group is residually finite, virtually residually
(finite $p$-)group for every sufficiently large prime $p$,  and coherent.
\end{theorem}

As an immediate corollary of Theorem \ref{thm3}, we deduce that one
cannot replace a multiplicative constant in the result of Baumslag-Miller-Trager \cite[Theorem 1]{BMT} mentioned above by
an additive constant: if $n\ge 3$, then there exists no maps
$\phi\colon F_n\to F_n$ such that $|\phi(R)|-|R|$ is bounded from
above by some constant $C$ and such that for every non-trivial $R\in
F_n$, the group $\langle F_n\mid \phi(R)=1\rangle$ is not residually
finite. Indeed, it is easy to see that if such a map exists, the
probability of a 1-related group with $n$ generators to be
residually finite would be bounded away from 1 as $|R|$ tends to
$\infty$.

Although the proof of Theorem \ref{thm3} presented below is relatively short, it uses some strong results from different areas of mathematics that are very rarely employed together: geometric group theory, algebraic geometry and probability theory (Brownian motion). Namely, we use the result of Feighn and Handel \cite{FH} that ascending HNN extensions of free groups are coherent (that is mostly geometric group theory, more precisely discrete Morse theory), results by Borisov and Sapir from \cite{BS}, \cite{BS2} that every ascending HNN extension of a free group is residually finite and even virtually residually (finite $p$-)group for almost all primes $p$ (that is essentially a result from algebraic geometry, namely the theory of quasi-fixed points of polynomial maps over finite fields), a result of Olshanskii \cite{Ol95} about subgroups with congruence extension property of hyperbolic groups (this is geometric group theory), a result of Kapovich, Schupp and Shpilrain \cite{KSS} generic solvability of isomorphism problem for 1-related groups (this result uses many parts of geometric group theory including the Arzhantseva-Olshansky method, boundaries of hyperbolic groups, etc.),
and a result by Cranston, Hsu, and March \cite{Cranston} that the boundary of the convex hull of a Brownian trajectory is smooth almost surely in Wiener's  measure (probability theory) .

Note that almost all (with probability tending
to 1) 1-related groups with a relator  of size $r>>1$ satisfy the small cancellation condition
$C'(1/6)$ and in fact $C'(\lambda)$ for every fixed $\lambda>0$ \cite{Gro}. Hence they are hyperbolic almost surely. It is still a major open
question in group theory whether every hyperbolic group is
residually finite. The positive answer would of course imply a part
of Theorem \ref{thm3}. Most specialists believe (see, for example, Gromov's conjecture \cite[5.B]{Gro}), however, that there are non-residually finite hyperbolic groups, and that, moreover, almost all hyperbolic groups and even most $C'(\lambda)$-groups are not residually finite. Constructing an example of a non-residually finite hyperbolic group is difficult because all non-elementary hyperbolic groups have ``very many" quotients, including torsion ad even, in most cases, bounded torsion, quotients \cite{Gro,Ol95}.

{\bf Acknowledgement} The authors are grateful to Vadim Kaimanovich, Ilya Kapovich , Yuval Peres, Vladimir Shpilrain and B\'alint Vir\'ag for their help.

\section{The theory of 1-related groups}

\subsection{The case of two generators}

Let $G=\langle a, b\mid R=1 \rangle$ be a 1-related group, $R$ is a cyclically reduced word in $F_2=\langle a,b\rangle$.
Consider a square lattice $\Gamma$ in $\R^2$, the Cayley graph of $\Z^2$. We assume that horizontal edges are labeled by $a$ and the vertical edges are labeled by $b$. Let $\psi:F_2\rightarrow\Z^2$ be the abelianization map. Let $w$ be the path in $\Gamma$ starting at the origin  $(0,0)$ and reading the word $R$.  This $w$ is called the {\em trace} of the relator $R$. Note that $w$ can visit every vertex (edge) many times. Vertices (edges) visited only once are called {\em simple}.
A line $L$ in $\R^2$ is said to be a \emph{supporting} line of $w$ if the path $w$ lies on one side of $L$ and has a common vertex with $L$.

\begin{theorem}(Brown \cite[Theorem 4.4]{Brown})\label{lm0}
Let $G = \langle a, b | R = 1\rangle$, where $R$ is a nontrivial cyclically reduced word in the free group on $\{a, b\}$ and $R\notin [F_2,F_2]$. Let $w$ be the trace of $R$, ending at a point $(m,n)$.

$G$ is an ascending HNN extension of a free group if and only if one of the two supporting lines of $w$ parallel to the vector $(m,n)$ intersects $w$ in one simple vertex or one simple edge.
\end{theorem}

\subsection{Embedding into 2-generated groups}

Let $G=\langle x_1,...,x_k\mid R=1\rangle$. If the sum of exponents of $x_i$ in $R$ is 0, then we can apply the Magnus rewriting to $R$. It consists of
\begin{itemize}
\item removing all occurrences of $x_i$ in $R$;
\item replacing every occurrence of a letter $x_j$ in $R$ by the letter $x_{j,p}$ where $p$ is the sum of exponents of $x_i$ in the prefix of $R$ before that occurrence of $x_j$.
\end{itemize}
Let $R'$ be the resulting word. The second indices $p$ of letters in $R'$ will be called the {\em Magnus $x_i$-indexes}. We say that certain Magnus index is {\em unique} if it occurs only once in $R'$.

We are going to use the following statement, which can be deduced from, say, a general result in \cite{Ol95} about hyperbolic groups.

\begin{lemma} \label{lm1} Let $w_1,...,w_k$ be words in the free group $F_n$ satisfying $C'(\frac1{12})$. Then the subgroup $H=\langle w_1,...,w_k\rangle$ of $F_n$ satisfies the {\em congruence extension property}, that is for every normal subgroup $N$ of $H$, the intersection of the
normal closure $N^G$ of $N$ in $F_2$ with $H$ is $N$. In particular, the natural homomorphism $H/N\to G/N^G$ is injective.
\end{lemma}

Let $\phi$ be the map $F_k\to F_n$ (where $F_n=\langle x_1,...,x_n\rangle$) given by $x_i\mapsto w_i, i=1,...,k$ where $w_1,...,w_k$ satisfy $C'(\frac1{12})$. Lemma \ref{lm1} immediately implies

\begin{lemma}\label{lm2} The map $\phi$ induces an injective homomorphism from $G=\langle x_1,...,x_k\mid R=1\rangle$ to the $1$-related $n$-generated group $\langle x_1,...,x_n\mid \phi(R)=1\rangle$.
\end{lemma}

\begin{theorem}\label{thm1}
Consider a group $G = \langle x_1,x_2,\dots,x_k | R = 1\rangle$, where $R$ is a
 word in the free group on $\{x_1,x_2,\dots,x_k\}$, $k\geq2$.
Assume the sum of exponents of $x_k$ in $R$ is zero and that the
maximal Magnus $x_k$-index of $x_1$ is unique. Then $G$ can be
embedded into an ascending HNN extension of a finitely generated
free group.
\end{theorem}

\begin{proof}
We may assume that the maximal Magnus $x_k$-index of $x_1$ is bigger than the one of $x_i$, for $1<i<k$, otherwise apply automorphism  $x_i\to x_k^{-m}x_i{x_k^m}, x_j\to x_j (j\ne i)$ for $m$ large enough.

Let $n\gg1$. Consider the following words $w_1,...,w_k\in F_2$.
\begin{align*}
w_1&=aba^2b...a^nba^{n+1}ba^{-n-1}ba^{-n}b...a^{-2}ba^{-1}b\\
w_i&=ab^ia^2b^i...a^nb^ia^{-n}b^i...a^{-2}b^ia^{-1}b^i, \quad \mathrm{ for}\quad 1<i<k\\
w_k&=ab^ka^2b^k...a^nb^ka^{-n}b^k...a^{-2}b^k
\end{align*}

These words satisfy the following conditions

\begin{enumerate}
\item[(1)] For a large enough $n$, these words and their cyclic shifts satisfy the small cancellation condition $C'(\frac1{12})$. Indeed, the maximal length of a subword repeating twice as a prefix of cyclic shifts of $w_i$ does not exceed $2n+3+k$, and the length of each $w_i$ is at least $n^2$. For a large enough $n$, we have $\frac{2n+3+k}{n^2}<\frac1{12}$.

\item[(2)] The sum of exponents of $a$ in $w_i$, $i<k$, is equal to 0, the sum of exponents of $a$ in $w_k$ is $1$.

\item[(3)] The maximal Magnus $a$-index of $b$ in $w_1$ is $\frac{(n+1)(n+2)}{2}$, and this index is unique. The maximal Magnus $a$-indices of $b$ in all other words are strictly smaller than the one in $w_1$.
\end{enumerate}

By Lemma \ref{lm2}, the group $G$ embeds into the 2-generated $1$-related group with presentation $\langle a,b\mid R(w_1,...,w_k)=1 \rangle$.

It remains to prove that $R(w_1,...,w_k)$ satisfies the conditions
of Lemma \ref{lm0}. Let $R'=R(w_1,...,w_k)$. Clearly the sum of
exponents of $a$ in $R'$ is zero. Every letter $b$ with maximal Magnus
$a$-index in $R'$ comes from some occurrence of a word $w_i$
substituted for letter $x_i$. The sum of exponents of $a$ is
nonzero only in $w_k$. Therefore the Magnus $a$-index of a letter $b$ is a
sum of the Magnus $x_k$-index of the letter $x_i$ in $R$, for which
it was substituted, and the Magnus $a$-index of $b$ in $w_i$. The
Magnus $x_k$-index in $R$ is maximal for the letter $x_1$ and the
maximum is unique in $R$.  The maximal Magnus $a$-index of $b$ in
$w_1$ is also unique (and bigger than in all other $w_i$'s). This
gives a uniqueness of  the maximal Magnus $a$-index in
$R'$. Therefore  there is a supporting line parallel to the  $b$-axes that
intersects the trace of $R'$ in one simple edge corresponding to the
letter $b$ with the maximal Magnus $a$-index. Therefore by Lemma \ref{lm0}
the group $\langle a,b|R'=1\rangle$ is an ascending HNN extension of
a finitely generated free group.
\end{proof}

\subsection{More than 2 generators and walks in $\Z^k$}

In the case of more than two generators we generalize the notion of
supporting line in the following way. Given a relator $R$, a
nontrivial 
word in the free group on $\{x_1,x_2,\dots,x_k\}$, let $w$ be its trace
in the lattice $\Z^k$.
For a letter $t\in\{x_1,x_2,\dots,x_k\}$, let $w_t$ be a set of
edges labeled by $t$ in $w$. A vertex on $w_t$ is called {\em
simple} if it does not belong to two edges of $w_t$. In particular,
if $w$ contains two consecutive edges with labels $t, t^{-1}$, then
the endpoints of these edges are not simple vertices.

\begin{definition}
A hyperplane $P$ is a \emph{supporting} hyperplane of $w_t$ if the trace
$w_t$ lies on one side of $P$ and has a common vertex with $P$. A hyperplane $P$
is said to be \emph{touching $w$} if
 \begin{itemize}
 \item $P$ is parallel to  the line containing the origin and the endpoint of $w$,
 \item there is $t\in\{x_1,x_2,\dots,x_k\}$, such that $P$ is a supporting hyperplane of $w_t$,
 \item the intersection of $P$ and $w_t$ consists of one simple vertex or one simple edge.
\end{itemize}
\end{definition}

\begin{lemma} \label{lm4}Let $G = \langle x_1,x_2,\dots,x_k | R = 1\rangle$, where $R$ is a
word in the free group on $\{x_1,x_2,$ $\dots,x_k\}$, $k\geq2$. Let $w$
be a trace of $R$ in the lattice $\Z^k$. If there is a hyperplane
$P$ touching $w$, then
 $G$ can be embedded into an ascending HNN extension of a free group.
\end{lemma}

\begin{proof}
We will embed $G$ into a one-relator group on $k+1$ generators that satisfies the condition of Theorem \ref{thm1}.

 If the normal vector of $P$ has irrational entries, then there is a
hyperplane $P'$ whose normal vector has rational entries that is
also touching $w$. Thus we can assume the normal vector of $P$ pointing toward
the half-space not containing $w_t$ is $(n(1),n(2),...,n(k))$ with
integer entries.

Consider the following substitution $\phi$:
\begin{align*}
x_i\mapsto x_iz^{n(i)},\quad i=1,\dots,k.
\end{align*}

Let $H=\langle x_1,x_2,\dots,x_k,z|\phi(R)\rangle$. Then $G$ is embedded into $H$ by $\phi$.
Since the normal vector of $P$ is orthogonal to the line connecting the origin and the endpoint of $w$, the sum of exponents of $z$ in $\phi(R)$ is zero.

It remains to show that the maximal Magnus $z$-index of $x_t$ in $\phi(R)$ is unique.

We can assume that the edge in $w_t$  intersecting $P$ corresponds to the first letter of $R$. 
Assume that there is another letter $x_t$ (at position $j$) in $\phi(R)$ with at least the same Magnus $z$-index as the first letter $x_t$ in the word.
Let $m(i)$ be the total sum of exponents of letter $x_i$ between these two occurrences of $x_t$ (note that it is the same in $R$ as in $\phi(R)$). If the exponent of the first letter $x_t$ is $1$, then add $1$ to $m(t)$.  If the exponent of the other letter $x_t$ (at position $g$) is $-1$, then subtract $1$ to $m(t)$. The Magnus $z$-index of the latter letter $x_t$ differs from the  Magnus $z$-index of the first letter by precisely $m(1)n(1)+m(2)n(2)+\dots+m(k)n(k)$.

Consider the edge corresponding to the first letter $x_t$ and the edge of letter $x_t$ at position $j$. Connect their initial points in $\Z^k$ by a vector (the vector connecting their terminal points is the same). It is easy to see that the coordinates of this vector are $(m(1),m(2),\dots,m(k))$. If the scalar product of this vector with the normal vector of $P$ is non-negative, then one of the endpoints of the edge of letter $x_t$ at position $g$ lies at $P$ or on the other side than $w_t$. This is impossible, because $P$ is a hyperplane touching $w$ (with respect to $x_t$).
\end{proof}

\begin{remark} Let $R$ be a non-reduced word in $\{x_1,...,x_n\}$,
and let $R'$ be the cyclically reduced form of $R$. Let $w, w'$ be
the walks corresponding to $R$ and $R'$ respectively. If there
exists a touching plane for $w$, then there exists a touching plane
for $w'$. The proof easily proceeds by induction on the number of
reductions.
\end{remark}

Let $w$ be the walk in $\Z^k$ corresponding to $R$. Let $\xi$ be the
vector connecting the initial and the terminal points of $w$. Let
$t\in \{1,...,k\}$. For every supporting plane $P$ of $w_t$ let $P^+$
be the closed half-space of $\R^k$ bounded by $P$ and containing
$w_t$. The intersection of all $P^+$ is a convex polyhedron in
$\R^k$. We shall call  $\Delta_0(t)$ the projection of the boundary
$\Delta(t)$ of that polyhedron onto the hyperplane orthogonal to
$\xi$. Then $\Delta(t)$ is the right cylinder with base
$\Delta_0(t)$, i.e. the direct product $\Delta_0(t)\times \R$. A
vertex of the random walk projected to a $0$-cell of $\Delta_0(t)$ is called a
{\em corner}
. For every vertex $x$ that is a
$0$-cell of a $\Delta_0(t)$, the line $x+\R\xi\subseteq \Delta(t)$
will be called the {\em support line} of $w_t$.

Lemma \ref{lm4}  immediately implies

\begin{lemma} \label{lm5} If one of the support lines of $w_t$ intersects $w_t$ in a simple vertex or a simple edge, then $G$ is embeddable into an ascending HNN extension of a free group.
\end{lemma}

\section{Random walks in $\Z^k$}
\subsection{Preliminaries}

Denote by $P^{NR}_n$ the (uniform) measure on simple random walks of length $n$ (not necessary reduced) and by $P^{NB}_n$ the uniform measure on non-backtracking simple random walks of length $n$.
To model cyclically reduced words, we denote by  $P^{CR}_n$ the uniform measure on non-backtracking simple random walks with last edge that is not inverse of the first edge of the walk (note that asymptotically this happens with probability $(2k-1)/2k$).
In all cases we can consider the sample space $\Omega$ containing all walks of any finite length.

We say that an event $A$ depends only on the cyclically reduced path of the random walk if $w\in A$ if and only if $w'\in A$, where $w'$ is the cyclically reduced path of $w$. An example of such event is an event that a support line of the cyclically reduced path $w'$ of a random walk $w$ intersects $w'$ in a simple vertex or a simple edge.

\begin{lemma} \label{lem2}
Let $A$ be an event depending only on the cyclically reduced path of the random walk. Assume  $\lim_{n \to\infty} P^{CR}_n(A)$ exists, then
\begin{align*}
    \lim_{n \to\infty} P^{NR}_n(A)=  \lim_{n \to\infty} P^{CR}_n(A).
\end{align*}
\end{lemma}

\begin{proof}
Let $\lim_{n \to\infty} P^{CR}_n(A)=a$ and assume $n_0$ is such that for all $n>n_0$
\begin{align*}
|P^{CR}_n(A)-a|<\epsilon.
\end{align*}

If an event $A$ depends only on the cyclically reduced path $w'$ of a random walk $w$, then conditioning on the length of the cyclically reduced path $|w'|$ we see that
$P^{NR}_n(A||w'|=k)=P^{CR}_k(A)$, provided $P^{NR}_n(|w'|=k)>0$. Let $n_1$ be such that for all $n>n_1$,
$P^{NR}_n(|w'|< n_0)\leq\epsilon.$

Then
\begin{align*}
P^{NR}_n(A)=\sum_{k=0}^n P^{NR}_n(|w'|= k)P^{CR}_k(A),
\end{align*}
and we can split the sum in two parts ($k\leq n_0$ and $k>n_0$) and obtain for $n>n_1$
\begin{align*}
(1-\epsilon)(a-\epsilon)<P^{NR}_n(A)<\epsilon+(a+\epsilon).
\end{align*}

Therefore $\lim_{n \to\infty} P^{NR}_n(A)=a$.
\end{proof}

\begin{lemma} \label{lem3}
Assume  $\lim_{n \to\infty} P^{CR}_n(A)$ exists, then the limit probability in the model IC exists as well and they are the same.
\end{lemma}

\begin{proof}
By the result of Kapovich, Schupp and Shpilrain \cite[Theorem C]{KSS}, there is a {\em generic} set of cyclically reduced words  $Q$ such that two 1-related groups with relators in $Q$ are isomorphic if and only if their relators are obtained from each other by a relabeling automorphism and cyclic shift. This set is  generic in the following sense: 
$$\lim_{n \to\infty} \frac{|T_n\cap Q|}{|T_n|}=1,$$
where $T_n$ consists of all cyclically reduced words in $F_k$ of length $n$.

Assume $\lim_{n \to\infty}P^{CR}_n(A)=p$. If we consider only words in $Q$, the limit clearly stays the same. Each word in $D_r$ represents several words in $Q\cup T_r$. The maximum number of words which may be relators of pairwise isomorphic 1-related groups is the number of relabeling automorphisms ($2^kk!$) times the number of cyclic shifts ($r$). There are words for which different cyclic shift are equal words after relabeling (these words are products of more than one copy of the same word in different alphabets), and there are words containing fewer than $k$  letters, but the probability of obtaining such a word tends to $0$ (exponentially) as the length grows. If we exclude such words from $Q\cap T_r$, then the set of groups given by relators in $Q\cap T_r$ contains groups isomorphic to those with a reltaor from $D_r$ (that is the set from the model IC), each with the same multiplicity. Therefore the limit probability in the model IC is equal to $p$ as in the model CR.
\end{proof}

Next we will need a modified version of The Donsker's invariance principle.
Denote by $C$ the space of all continuous function $f:[0,1]\to \R^k$ such that $f(0)=0$, equipped with the sup norm.

\begin{theorem}[Donsker's Theorem modified]\label{donsker} Consider a piecewise linear function $Y_n(t):[0,1]\to \R^k$, where the line segments are connecting points
$Y_n(t)=S_{nt}/\sqrt{n}$ for
$t=0,1/n,2/n$, $\dots,n/n=1$, where $(S_n)$ has a distribution according to $\P^{CR}_n$.
Then $Y_n(t)$ converges in distribution to a Brownian motion, as $n \to \infty$.
\end{theorem}

\begin{proof} First we prove that conditioning on the first step of non-backtracking random walk has asymptotically no influence on $Y_n(t)$, which allows us to switch between $P^{NB}$ and $P^{CR}$. Next, we basically repeat the proof of the Donsker's Theorem in \cite[Theorem 10.1]{Billingsley}. The Central Limit Theorem for non-backtracking walks that we will use was proved in \cite{Rivin}.

Let $(R_n)$ be a non-backtracking random walk. We cut the walk at time $\ln(n)$, splitting the walk into two (dependent) parts $\left(R^{(1)}_{\ln(n)}\right)$ and $\left(R^{(2)}_{n-\ln(n)}\right)$. Define piecewise linear functions $X(t)$ and $Z(t)$ connecting points $X(t)=R_{nt}/\sqrt{n}$ and $Z(t)=R^{(2)}_{nt}/\sqrt{n-\ln(n)}$ respectively. Clearly, the distance (in the sup norm) between $X(t)$ and $Z(t)$ goes to $0$, as $n \to\infty$. Moreover the latter part of the walk $\left(R^{(2)}_{n-\ln(n)}\right)$ tends to be independent of the first step of $(R_n)$, as $n\to \infty$. Therefore the  piecewise linear functions obtained from walks with measures $\P^{NB}$ and $\P^{CR}$ have the same limit in distribution.

Next we show that the finite-dimensional distribution of $Y_n(t)$ converges to the one of Brownian motion. By the result of Rivin \cite[Theorem 5.1]{Rivin} the probability distribution of $S_{n}/\sqrt{n}$ converges to a normal distribution on $\R^k$, whose mean is $0$ and covariance matrix is diagonal, with entries
\begin{align*}
\sigma^2=\frac{1}{\sqrt{2k-1}}\left[ 1+\left(\frac{c+1}{c-1}\right)^{1/2}\right],
\end{align*}
where $c=k/\sqrt{2k-1}$. By the previous paragraph this holds for  $R_{n}/\sqrt{n}$ as well.

Consider now the two-dimensional distribution, that is the position at two time points, $s<t$. It is enough to show that $S_{ns}/\sqrt{n}$ and $(S_{nt}-S_{ns})/\sqrt{n})$ are asymptotically independent (the normal distribution of each of them was already established). The first step of  $(S_{nt}-S_{ns})$ is not independent of $(S_{ns})$, but asymptotically the distribution of $(S_{nt}-S_{ns})/\sqrt{n}$ is independent of the first step. The convergence of finite-dimensional distribution for more time points can be proved in the same way.

It remains to show the tightness of the process. We refer ourselves to the proof in Billingsley \cite[Page 69]{Billingsley}, and here we prove only the lemma needed. The claim is:
\begin{align*}
\P^{NB}_n\left( \max_{i<n}|S_i|\geq\lambda\sigma\sqrt{n}\right)\leq
\P^{NB}_n\left(|S_n|\geq(\lambda-\sqrt{2})\sigma\sqrt{n} \right).
\end{align*}
In order to prove this, we define events $E_i=\left\{\max_{j<i}|S_j|<\lambda\sigma\sqrt{n}\leq|S_i|\right\}$.
Now we have:
\begin{align*}
\P^{NB}_n\left( \max_{i<n}|S_i|\geq\lambda\sigma\sqrt{n}\right)&\leq
\P^{NB}_n\left( |S_n|\geq(\lambda-\sqrt{2})\sigma\sqrt{n}\right)\\
&\quad +
\sum^{n-1}_{i=1} \P^{NB}_n\left( E_i \cap
 \left\{|S_n|<(\lambda-\sqrt{2})\sigma\sqrt{n}\right\}\right),
\\
\P^{NB}_n\left( E_i \cap
 \left\{|S_n|<(\lambda-\sqrt{2})\sigma\sqrt{n}\right\}\right)&\leq
 \P^{NB}_n\left( E_i \cap
 \left\{|S_n-S_i|\geq\sigma\sqrt{2n}\right\}\right)\\
 & =\P^{NB}_n( E_i) \P^{NB}_n
 \left(|S_n-S_i|\geq\sigma\sqrt{2n}\right).
\end{align*}
The last equality follows from the fact that the length of $S_n-S_i$ is independent of the walk up to the time $i$. Now by Chebyshev's inequality $\P^{NB}_n
 \left(|S_n-S_i|\geq\sigma\sqrt{2n}\right) \leq 1/2$. The claim follows from
 \begin{align*}
 \sum^{n-1}_{i=1} \P^{NB}_n\left( E_i \cap
 \left\{|S_n|<(\lambda-\sqrt{2})\sigma\sqrt{n}\right\}\right)&\leq
 \frac{1}{2}\sum^{n-1}_{i=1} \P^{NB}_n\left( E_i\right)\\
 &\leq\frac{1}{2} P^{NB}_n\left( \max_{i<n}|S_i|\geq\lambda\sigma\sqrt{n}\right).
 \end{align*}\end{proof}

\subsection{Corners of random walk}

Let $(S_n)$ be a non-backtracking random walk in $\Z^k$ of length $n$ with last edge that is not inverse of the first edge (according to the measure $P^{CR}_n$).
Recall that $\Delta_0$ is the projection of the boundary of the
convex hull of the random walk $(S_n)$ onto the hyperplane orthogonal to
$\xi$, the vector connecting the initial and the terminal points of
the random walk. Denote by $H_n$, the set of corners, which are the vertices
of the random walk that project to the $0$-cells of $\Delta_0$. We count the
corners with their multiplicities.

\begin{lemma} Let $(S_n)$ be a non-backtracking random walk in $\Z^k$ of length $n$ with last edge that is not inverse of the first edge (according to the measure $P^{CR}_n$).
Let $H_n$ be the set of its corners as defined above.
Then for any integer $m$
\begin{align*}
\P^{CR}_n(|H_n|<m)\to 0\textrm{ as }n\to \infty.
\end{align*}
\end{lemma}

\begin{proof}
Consider a piecewise linear function $X_n(t):[0,1]\to \R^k$, where the line segments are connecting points $X_n(t)=S_{nt}/\sqrt{n}$ for $t=0,1/n, 2/n,\dots,$ $n/n=1$.
Recall that $C$ is the space of all continuous function $f:[0,1]\to \R^k$ such that $f(0)=0$, equipped with the sup norm. By Theorem \ref{donsker}, $X_n(t)$ converges in distribution to a Brownian motion, as $n \to \infty$.
Denote by $A_m$ a subset of $C$ such that $f\in A_m$ if  the convex hull of the projection of $f$ to a hyperplane orthogonal to $f(1)$ is a $(k-1)$-dimensional (convex)  polytope with at most $m$ $0$-cells. We will show that the set $A_m$ is a closed subset of $C$ in the sup norm and that the Wiener measure of $A_m$ is zero. It follows from the weak convergence that $\P (X_n(t)\in A_m)\to 0$ as $n\to \infty$.

First show that $A_m$ is closed. Let $f\notin A_m$ be a limit (in the sup norm) of $f_n\in A_m$.
Let $p_n$ (resp. $p$) denotes a projection on a hyperplane at the origin orthogonal to $f_n(1)$ (resp. $f(1)$). For any $\epsilon>0$ and for all but finitely many $n$, we have $|p(f(t))-p(p_n(f_n(t)))|<\epsilon$  for all $t$. If the convex hull of $p_n(f_n)$ is a polytope with at most $m$ $0$-cells, then the same holds for $p(p_n(f_n))$.
Denote by $B_n$ (resp. $B$) the convex hull of $p(p_n(f_n))$ (resp. $p(f)$).
Then for any $\epsilon>0$ the boundary of $B$ is in Hausdorff $\epsilon$-neighborhood of the boundary of $B_n$ for all but finitely many $n$. We need to prove that if a convex body in $\R^k$ is
arbitrarily close to some polytope with at most $m$ $0$-cells, then the body itself is such a polytope.
To prove that, enumerate the $0$-cells of $B_n$ somehow $\{v_{n,1},...,v_{n,m}\}$ (the last few $0-$cells may coincide if the total number of $0-$cells is smaller than $m$). Choose one convergent subsequence $\{v_{n_j(i),i}\}_{j=1,2,...}$ of $\{v_{n,i}\}$ for each $i$ in such a way that the set $\{n_j(i+1), j=1,2,...\}$ is a subset of $\{n_j(i), j=1,2,...\}$. Let $N^{(0)}$ be the sequence $\{n_j(m), j=1,2,...\}$. There exists a subset $N^{(1)}$ of $N^{(0)}$ such that for every $i,j\in \{1,2,...,m\}$ either $v_{t,i}$ and $v_{t,j}$ span a 1-cell in all $B_t$, $t\in N^{(1)}$ or they don't span a 1-cell in all $B_t$, $t\in N^{(1)}$. Proceeding by induction on the dimension of a cell, we can find an infinite subset $N$ of natural numbers such that for every subset $M\subset \{1,2,...,m\}$ either vertices $v_{t,i}$, $i\in M$, span a cell in $B_t$ for all $t\in N$ or they span a cell in none of these $B_t$. For every $M\subseteq \{1,2,...,m\}$ such that $v_{t,i}$, $i\in M$, span a cell $F_M(t)$ of dimension $j$ in all $B_t, t\in N$, the limit $\lim_{t\in N} F_M(t)$
exists and is an Euclidean convex polytope of dimension $j$ spanned by the 0-cells $v_i, i\in M$. Hence the convex hull $B$ of $p(f)$ is a convex polytope with at most $m$ 0-cells.
To see that $A_m$ has measure $0$ we introduce the following set $D$ of continuous functions $[0,1]\to \R^k$. A function $f\in C$ is in $D$ if a convex hull of its projection to some $2$-dimensional plane orthogonal to $f(1)$ has a smooth boundary, i.e. it is a $C^1$ curve in the plane. Clearly $D\cap A_m$ is empty for all $m$.
Let $X_t$ be a standard Brownian motion in $\R^k$. Then $X(t)-t X(1)$ is a Brownian bridge in $\R^k$. All projections of this Brownian bridge to $\R^2$ are equivalent in distribution and give Brownian bridges in $\R^2$.
To conclude that $D$ has Wiener measure $1$, it is enough to show that the convex hull of a planar Brownian bridge has a smooth boundary almost surely.  For Brownian motions, that is proved in  \cite{Cranston}. We are going to use almost the same argument.

Consider a Brownian bridge and pick any of its extreme points. Move the beginning of the time from $0$ to this extreme point and rotate the plane so that the path is in the upper half plane. The obtained process $Y_t$ is a Brownian excursion, i.e. it stays in the upper half plane and return to the starting point. The same is true for $Y_{1-t}$. Then the transformation $V_t:=(1+t)Y_{t/(1+t)}$ is a Brownian meander (see for example \cite{Billingsley}, p.68, exercise 3).  Let $V_t=(V_t(1),V_t(2))$, by \cite{Burdzy}, for any $c>0$, we have
\begin{align*}
\P(\inf\{t:t>0,|V_t(2)|\leq c|V_t(1)|\}=0)=1.
\end{align*}
By reversing the transformation we obtain the same property for $Y_t$ and $Y_{1-t}$. Now the claim follows using the argument from Theorem 1 in \cite{Cranston}.
\end{proof}

We say that a random walk is \emph{bad} if there is no $0$-cell of $\Delta_0$ such that only a single vertex 
is projected to it.

\begin{lemma}\label{corners}
The probability that a $k$-dimensional non-backtracking simple random walk (with last edge that is not inverse of the first edge of the walk) is bad in the above sense tends to $0$, for $k>2$.
\end{lemma}

\begin{proof}
Let $(S_n)$ be a $k$-dimensional non-backtracking simple random walk (with last edge that is not inverse of the first edge of the walk), $k>2$.  The number of all cyclically reduced walks, $|T_n|$, equals asymptotically $(2d-1)^{n}$.  Let $B_n$ be a set of all ``bad'' walks, i.e. for all $0$-cells of $\Delta_0$ we have at least two vertices projected to it.

Define a map $\tau_i:B_{n}\to T_{n+4}$ that inserts a commutator at an $i$-th corner of the random walk in such a way that it produces a new corner. For example, if the corner is between letters $x_ix_j$, we can insert $x_ix_j^{-1}x_i^{-1}x_j$ in between them, so that the second vertex of these three new vertices projects outside of $\Delta_0$ of the original walk. Note that the new walk is not bad anymore.

This map $\tau_i$ is injective. Moreover images of the same walk under $\tau_i$ for different $i$ are disjoint.
The set of bad walks with more than $K$ corners, $U_K:=\{w| w\in B_n, H_n(w)>K\}$, is mapped by $\tau_1, \tau_2,\dots \tau_K$ into $T_{n+4}$. The union of their images $\bigcup_{i=1}^K\tau_i(U_K)$ is of size $K|U_K|$.

For any integer $K$ we can write:
\begin{align*}
|B_{n}|&=|B_n\setminus U|+|U|\\
&\leq \P(H_n<K)|T_{n}|+\frac{|T_{n+4}|}{K}\\
\frac{|B_{n}|}{|T_{n}|}&\leq \P(H_n<K) + \frac{1}{K}\frac{|T_{n+4}|}{|T_{n}|}\\
\frac{|B_{n}|}{|T_{n}|}&\leq\P(H_n<K)+\frac{(2d-1)^4}{K} \to \frac{(2d-1)^4}{K}
\end{align*}
This holds for $K$ arbitrarily large. The first summand tends to $0$, as $n\to\infty$, by Lemma \ref{corners}. It implies that the probability of a bad walk is less than any positive number.
\end{proof}

\begin{remark} Using Lemma \ref{lem2}, the same result as in Lemma \ref{corners} follows for a simple random walk (we consider corners of its reduced form).
\end{remark}

\begin{proof}[Proof of Theorem \ref{thm3}]
Let $G = \langle x_1,x_2,\dots,x_k | R = 1\rangle$ be a random $k$-generator
$1$-relator group, $k>2$.
If the trace of $R$ is not bad in the above sense, then there is a hyperplane touching its cyclic reduction.
Thus, by Lemma \ref{corners}, there is a hyperplane touching $R$ with probability tending to $1$. By Lemma \ref{lm4}, this implies that the group can be embedded into an ascending HNN extension of a free group almost surely.
\end{proof}

\section{What if the number of generators is 2? Some open questions}

Theorem \ref{thm3} leaves the case of 1-related groups with 2 generators open. The reason is the following. 
In \cite{BS}, Borisov and Sapir reported a result of computations
saying that, apparently, more than 94\% of 1-related groups with 2
generators and a relator of size $n>>1$ are ascending HNN extensions
of free groups .  Borisov and Sapir used the Monte-Carlo method for
$n\approx 10^6$. Schupp and later Dunfield and Thurston \cite{DT}
conducted  similar experiments on their own and came to the same
conclusion. Thus in some sense we know (although it is not proved yet) that probably majority of 1-related groups with 2 generators are residually finite and coherent. At the same time, 
Dunfield and Thurston noticed \cite{DT} that a 2-generated 1-related group is
not almost surely an ascending HNN extension of a free group (that
is the probability that a 2-generated 1-related group with a relator
of size $n$ is an ascending HNN extension of a free group is bounded
away from 1 as $n\to \infty$).

Nevertheless the answer to the following question can be positive

\begin{question}\label{q1} Is it true that the probability that a 2-generated 1-related group is residually finite is 1?
\end{question}

Answering a question of M. Sapir, Fu and Virag proved (see \cite{Fu}) that a 2-generated  1-related group is almost surely a very special HNN extension of a free group. That HNN extension is determined by three parameters: integer $k$ (the rank of the free group), integer $i$ between 1 and $k$, and a word $w$ from the free group $F_k$. It is given by the following presentation:
\begin{align*}H(k,i,w) =  & \langle a_1,...,a_k, t\mid
ta_1t^{-1} = a_2,...,ta_{i-1}t^{-1}=a_i, ta_it^{-1}= w, twt^{-1} = a_{i+1},\\
&  ta_{i+1}t^{-1} = t_{i+2}, ..., ta_{k-1}t^{-1} = a_k\rangle.
\end{align*}

\begin{question}\label{q2} Is every group $H(k,i,w)$ residually finite?
\end{question}

By \cite{Fu}, positive answer to Question \ref{q2} implies positive answer to Question \ref{q1}.

\def\cprime{$'$}

\comment

\begin{minipage}[t]{3 in}
\noindent Mark V. Sapir\\ Department of Mathematics\\
Vanderbilt University\\
m.sapir@vanderbilt.edu\\
www.math.vanderbilt.edu/$\sim$msapir\\
\end{minipage}
\begin{minipage}[t]{2.4 in}
\noindent Iva \v Spakulov\' a\\ Department of Mathematics\\
Vanderbilt University\\
iva.spakulova@gmail.com\\
\end{minipage}
\endcomment

\end{document}